\documentclass{conm-p-l}

\usepackage{amsthm,amsfonts, cite}
\usepackage{amssymb,graphicx,color, tikz-cd, tikz,float, footmisc}
\usepackage[all]{xy}
\usepackage{eucal}
\usepackage{enumerate}
\usepackage{hyperref}
\numberwithin{equation}{section}

\theoremstyle{plain}
\newtheorem{teo}{Theorem}[section]
\newtheorem*{teo*}{Theorem}
\newtheorem{teoA}{Theorem}

\newtheorem{cor}[teo]{Corollary}
\newtheorem*{cor*}{Corollary}
\newtheorem{lem}[teo]{Lemma}
\newtheorem*{lem*}{Lemma}
\newtheorem{prop}[teo]{Proposition}
\newtheorem*{prop*}{Proposition}

\theoremstyle{remark}
\newtheorem{obs}[teo]{Remark}
\newtheorem{ex}[teo]{Example}

\newcommand{\R}{\ensuremath{{\mathbb{R}}}}
\newcommand{\C}{\ensuremath{{\mathbb{C}}}}

\newcommand{\g}{\ensuremath{\mathrm{g}}}

\renewcommand\vec[1]{\boldsymbol{#1}}

\newcommand\Y{\ensuremath{\mathbf{Y}}}

\makeatletter
\@namedef{subjclassname@2020}{\textup{2020} Mathematics Subject Classification}
\makeatother

\begin{document}

\title[The curvature of left-invariant magnetic systems]{The curvature of left-invariant magnetic systems}

\author[S. Rankin]{Shane Rankin}
\author[I. Terek]{Ivo Terek$^\dagger$\footnote{$\dagger$ \emph{Corresponding author}.}}
\author[D. Weed]{David Weed}

\address{\parbox{\linewidth}{Department of Mathematics \\ University of California, Riverside \\ Riverside, CA 92521, USA \\[-.5em] }}

\email{shane.rankin@email.ucr.edu}
\email{ivo.terek@ucr.edu}
\email{david.weed@email.ucr.edu}

\makeatletter
\@namedef{subjclassname@2020}{\textup{2020} Mathematics Subject Classification}
\makeatother


\begin{abstract}
We compute the magnetic curvatures (in the sense of Assenza) of left-invariant magnetic systems on Lie groups and explore relations between curvature properties and algebraic properties, extending to the magnetic setting a number of results from Milnor's classic 1976 paper on the geometry of left-invariant metrics. We also discuss the existence of non-trivial bi-invariant magnetic systems and exhibit the resulting curvature formulas in some concrete examples, including the Heisenberg group.
\end{abstract}

\maketitle

\section{Introduction}

A \emph{magnetic system} on a smooth manifold $M$ is defined, as usual, to be a pair $(\g,\sigma)$, where $\g$ is a Riemannian metric and $\sigma$ is a closed $2$-form on $M$; in this context, $\sigma$ is called the \emph{magnetic $2$-form}. One may then consider the second-order ordinary differential equation
\begin{equation}\label{eqn:Landau-Hall}
  \frac{{\rm D}\dot{\gamma}}{{\rm d}t}(t) = \Y_{\gamma(t)}(\dot{\gamma}(t)),
\end{equation}for smooth curves $\gamma\colon I\to M$, where the covariant derivative ${\rm D}/{\rm d}t$ is induced by the Levi-Civita connection of $\g$, and the \emph{Lorentz force operator} $\Y\colon TM\to TM$, skew-adjoint and equivalent to $\sigma$ under $\g$, is characterized by
\begin{equation}\label{eqn:Lorentz_force}
  \sigma_x(v,w) = \g_x(\Y_x(v),w),\quad\mbox{for all $x\in M$ and $v,w\in T_xM$}.
\end{equation}Equation \eqref{eqn:Landau-Hall} is sometimes called the \emph{Landau-Hall equation}, and its solutions are called $(\g,\sigma)$-geodesics; they all have constant speed due to skew-adjointness of $\Y$, and $(\g,\sigma)$ is called \emph{complete} when all inextendible $(\g,\sigma)$-geodesics are defined on all of $\R$. In this case, the magnetic flow obtained from \eqref{eqn:Landau-Hall} may be restricted to each energy level $\Sigma_s = \{(x,v) \in TM :  \|v\|_x = s\}$, where $s>0$. The dynamical and geometric properties of such magnetic flows, that now strongly depend on the value of $s$ (unlike in the classical Riemannian case, where $\sigma=0$), have been---and continue to be---largely investigated. See for instance \cite{Burns-Paternain_2002, Merry_2010, Wojtkowski_2000,BT98,Gouda-97} and the references therein, or \cite{AMMP_2017, ABM_2025, AB_2021, BK_2022, Maier_2026} for more recent work.

With the recent introduction by Assenza \cite{Assenza_2024} of a \emph{magnetic curvature operator}, and the companion notions of \emph{magnetic sectional and Ricci curvatures}, many results in Riemannian geometry have been generalized to the setting of magnetic systems. For example, we have the Bonnet-Myers and Synge's theorems in \cite{Assenza_2024}, Green and Hopf's theorems in \cite{ART_2024}, compatibility equations from submanifold theory have been listed in \cite{AGAG-26}, and an interpretation of magnetic curvature as the curvature of a certain Cartan connection is given in the recent preprint \cite{beaufort2026tensortomographyframeflow}. We review the precise definitions of magnetic curvature, all of which explicitly carry the parameter $s>0$ from the previous paragraph, in Section \ref{sec:mag-curvature-defn}.

Here, we consider the simplest prototypes of homogeneous spaces: Lie groups. Naturally, a magnetic system on a Lie group is \emph{left-invariant} if all left-translations are \emph{magnetomorphisms} (that is, diffeomorphisms preserving both the metric and magnetic $2$-form, cf. \cite{ABM_2025}). Explicit descriptions of magnetic geodesics have been obtained in Heisenberg groups \cite{EGM_2021, MN_2022}, and in three-dimensional groups with bi-invariant metrics \cite{Turhan_2020}, while integrability of magnetic flows (both in the Liouville and non-commutative senses) on Lie groups and homogeneous spaces has been thoroughly studied in \cite{Efimov_2005, BJ_2008, MSY_2008, DGJ_2025, Moroianu-Ovando}. Our goal here is to study curvature. Following in spirit parts of the presentation by Milnor \cite{Milnor_1976}, we register formulas for the magnetic curvatures of left-invariant magnetic systems in Section \ref{sec:inv-mag-curv}, and discuss some relations between curvature properties and algebraic ones.

Our main results follow the general trend of many rigidity-related theorems for magnetic systems, where the magnetic $2$-form ultimately vanishes; see e.g. \cite[Lemma A.7]{Paternain_2007}, \cite[Corollary 1.6]{mmls}, \cite[Theorem A and Corollary 3.6]{SIGMA-26}, and \cite[Theorem D]{ART_2024}. The first one we generalize is a well-known consequence of \cite[formula (7.3)]{Milnor_1976}, about the sectional curvature of a bi-invariant metric:
\begin{teoA}\label{thm:bi-inv-pos-sec}
  For any Lie group $G$ and bi-invariant magnetic system $(\g,\sigma)$ on $G$, and any $s>0$, it holds that ${\rm Sec}^{(\g,\sigma,s)} \geq 0$. Moreover, ${\rm Sec}^{(\g,\sigma,s)} = 0$ if and only if $G$ is Abelian and $\sigma=0$.
\end{teoA} 
As in \cite{Milnor_1976}, Theorem \ref{thm:bi-inv-pos-sec} follows from a more general computation---summarized in Proposition \ref{prop:sec-central}---of the $s$-magnetic sectional curvature along directions which are ``compatible'' with $(\g,\sigma)$, in the sense of Proposition \ref{prop:bi-inv-char}.

Next, we have a magnetic version of \cite[Theorem 1.5]{Milnor_1976}, which specializes \cite[Theorem D]{ART_2024} to the left-invariant setting:
\begin{teoA}\label{teo:Milnor-Flat}
  Let $(\g,\sigma)$ be a left-invariant magnetic system on a connected Lie group $G$. Then, there is $s>0$ such that ${\rm Sec}^{(\g,\sigma,s)}=0$ if and only if $\nabla\sigma=0$ and one of the following occurs:
  \begin{enumerate}[\normalfont(i)]
  \item $\sigma=0$ and the Lie algebra of $G$ splits as an orthogonal direct sum $\mathfrak{g} = \mathfrak{b}\oplus \mathfrak{u}$, where $\mathfrak{b}$ is a commutative subalgebra, $\mathfrak{u}$ is a commutative ideal, and ${\rm ad}_b$ is skew-adjoint for every $b\in\mathfrak{b}$.
  \item $\sigma$ is nowhere-vanishing and $J=\|\Y\|^{-1}\Y$ makes $(G,\g)$ a K\"{a}hler manifold holomorphically isometric to a complex hyperbolic space $\C\mathbb{H}^n$.
  \end{enumerate}
\end{teoA}

The main ingredient in the proof of Theorem \ref{teo:Milnor-Flat}, perhaps hinted at by item (ii) therein, is a classical result by Heintze: every connected homogeneous K\"{a}hler manifold of negative curvature is holomorphically isometric to a complex hyperbolic space \cite[Theorem 4]{Heintze_1974}.

Finally, we also register a magnetic version of \cite[Theorem 1.6]{Milnor_1976}:
\begin{teoA}\label{teo:solvable}
  Let $G$ be a connected Lie group, and $(\g,\sigma)$ be a left-invariant magnetic system on $G$. If there is $s>0$ such that ${\rm Sec}^{(\g,\sigma,s)} \leq 0$, then $G$ is solvable. If, in addition, $G$ is unimodular, then $\sigma=0$ and $\g$ is flat.
\end{teoA}

Briefly, the proof of Theorem \ref{teo:solvable} consists in noting that $\g$ must then necessarily be negatively-curved, and then applying \cite[Theorem 1.6]{Milnor_1976} itself; the details are given in Section \ref{sec:inv-mag-curv}. We observe, however, that the proof of \cite[Theorem 1.6]{Milnor_1976} largely relies on a geometric classification of negatively-curved homogeneous spaces. The difficulties with obtaining such a classification for negatively-curved magnetic systems seem to ultimately stem from the general failure of the Hopf-Rinow theorem in the magnetic setting \cite[Remark 2.5(b)]{BRCF_2005}, causing a crucial step in the proof of the corresponding would-be magnetic Cartan-Hadamard theorem to fail; without it, it is not clear how to trace out the extensive work of Wolf \cite{Wolf_1964}, Azencott and Wilson \cite{AZ1,AZ2}, Kobayashi \cite{Kobayashi_1962}, and others in order to obtain the desired classification.

\medskip

\noindent {\bf Organization of the paper.} Throughout the entire text, unless otherwise stated, we work in the smooth category, $G$ denotes a connected Lie group, and $\mathfrak{g}$ is its Lie algebra. In Sections \ref{sec:mag-curvature-defn} and \ref{sec:Riem-Lie-gps} we review basic facts about magnetic curvature and left-invariant metrics. In Section \ref{sec:inv-mag-curv} we put such facts together and register in Proposition \ref{prop:left-invariant-mag-sec} an explicit formula for the $s$-magnetic sectional curvature of left-invariant magnetic systems; a few consequences---including Theorems \ref{thm:bi-inv-pos-sec} and \ref{teo:solvable}---are then obtained. Next, in Section \ref{sec:non-triv-bi} we discuss the existence of nontrivial bi-invariant magnetic systems, and in Section \ref{sec:left-inv-flatness} we prove Theorem \ref{teo:Milnor-Flat}. Finally, in Section \ref{sec:Heisenberg} we compute the magnetic curvatures of a special class of systems on Heisenberg groups. Appendix \ref{app:EA} derives a generalization of the well-known Euler-Arnold formalism, for left-invariant exact magnetic systems.

\medskip

\noindent{\bf Funding statement.} No funding was received for this work.

\section{Magnetic curvature}\label{sec:mag-curvature-defn}

We briefly review, for the reader's convenience, the definitions of magnetic curvature presented by Assenza \cite{Assenza_2024}. See also \cite[pp. 2--3]{ART_2024}, \cite[p. 2]{SIGMA-26}, and \cite[p. 4]{AGAG-26}.

Let $(\g,\sigma)$ be a magnetic system on a smooth manifold $M$, and let $UM \to M$ be the unit-tangent bundle of $(M,\g)$. We consider two auxiliary bundles: the Stiefel bundle ${\rm St}_2(M,\g)\to M$ of orthonormal $2$-frames tangent to $M$, and the bundle $E\to UM$ of orthogonal hyperplanes tangent to $M$; specifically, their fibers are given by ${\rm St}_2(M,\g)_x = \{(v,w)\in (U_xM)^{\times 2} : \g_x(v,w)=0\}$ and $E_{(x,v)} = v^\perp$.

For each $s>0$, the \emph{$s$-magnetic curvature operator} of $(\g,\sigma)$ is the self-adjoint endomorphism $\vec{M}^{(\g,\sigma,s)}\colon E\to E$ given by
\begin{equation}
  \vec{M}^{(\g,\sigma,s)}_{(x,v)}(w) = R^{(\g,\sigma,s)}_{(x,v)}(w) + A^{(\g,\sigma)}_{(x,v)}(w),
\end{equation}where $R^{(\g,\sigma,s)},A^{(\g,\sigma)}\colon E\to E$ are in turn defined as
\begin{equation}\label{eqn:A-R-defs}
  \begin{split}
    A^{(\g,\sigma)}_{(x,v)}(w) &= -\frac{3}{4} \Y_x(\Y_x(w)^\top) -\frac{1}{4} \Y_x^2(w)^\perp,\quad\mbox{and} \\
    R^{(\g,\sigma,s)}_{(x,v)}(w) &= s^2 R^\g_x(w,v)v - s (\nabla_w\Y)(v) + \frac{s}{2} (\nabla_v\Y)(w)^\perp ,
  \end{split}
\end{equation}where $R^\g$ is the Riemann curvature tensor of $\g$, and the tangent and perpendicular projections are taken relative to the orthogonal splitting $T_xM = \R v \oplus v^\perp$ induced by $\g$. The general formulas
  \begin{equation}\label{eqn:Aww}
    \begin{split}
      \g_x(A^{(\g,\sigma)}_{(x,v)}(w),w) &= \frac{3}{4} \g_x(\Y_x(v),w)^2 + \frac{1}{4} \|\Y_x(w)\|^2  \\ \mbox{and}\quad {\rm tr}\big(A^{(\g,\sigma)}_{(x,v)}\big) &= \frac{1}{2} \|\Y_x(v)\|^2 + \frac{1}{4} \|\Y_x\|^2,\quad \mbox{where } \|\Y_x\| = \sqrt{-{\rm tr}(\Y_x^2)},
    \end{split}
  \end{equation}
clear from \eqref{eqn:A-R-defs}, will be useful in Sections \ref{sec:inv-mag-curv} and \ref{sec:Heisenberg}. The \emph{$s$-magnetic sectional curvature} ${\rm Sec}^{(\g,\sigma,s)}\colon {\rm St}_2(M,\g) \to \R$ and \emph{\mbox{$s$-magnetic} Ricci curvature} ${\rm Ric}^{(\g,\sigma,s)}\colon UM\to \R$ are defined by
\begin{equation}
  {\rm Sec}^{(\g,\sigma,s)}_x(v,w) = \g_x\big(\vec{M}^{(\g,\sigma,s)}_{(x,v)}(w),w\big)\quad\mbox{and}\quad {\rm Ric}^{(\g,\sigma,s)}(x,v) = {\rm tr}\,\big(\vec{M}^{(\g,\sigma,s)}_{(x,v)}\big).
\end{equation}The $s$-magnetic sectional curvature does not survive as a function on the Grassmannian bundle of $2$-planes tangent to $M$ with,
\begin{equation}\label{eqn:sec-not-symmetric}\parbox{.53\textwidth}{in general, ${\rm Sec}^{(\g,\sigma,s)}_x(v,w) \neq {\rm Sec}^{(\g,\sigma,s)}_x(w,v)$.}\end{equation} As in the Riemannian case, however, the $s$-magnetic Ricci curvature may be expressed in terms of the $s$-magnetic sectional curvature, with
\begin{equation}\label{eqn:mag-Ric}
  {\rm Ric}^{(\g,\sigma,s)}(x,v) = \sum_{i=1}^{n-1} {\rm Sec}^{(\g,\sigma,s)}_x(v,e_i),
\end{equation}where $e_1,\ldots, e_{n-1}$ is any orthonormal basis of $v^\perp$. Some useful relations between the magnetic curvatures and the Riemannian ones are also readily obtained, e.g.
\begin{equation}\label{eqn:mag-sec-quadratic}
  {\rm Sec}^{(\g,\sigma,s)}_x(v,w) = s^2 {\rm Sec}^\g_x(v,w) - s\g_x((\nabla_w\Y)(v),w) + \g_x\big( (A^{(\g,\sigma)}_{(x,v)}(w),w\big),
\end{equation}cf. \cite[Formula (6.3)]{ART_2024}, and from there with \eqref{eqn:mag-Ric},
\begin{equation}\label{eqn:mag-Ric-div}
  {\rm Ric}^{(\g,\sigma,s)}(x,v) = s^2 {\rm Ric}^\g_x(v,v) - s\,({\rm div}_\g \sigma)_x(v) + {\rm tr}\big(A^{(\g,\sigma)}_{(x,v)}\big),
\end{equation}where ${\rm div}_{\g}\sigma \in \Omega^1(M)$ is given by $({\rm div}_\g\sigma)(Y) = {\rm tr}_{\g}((X,Z) \mapsto (\nabla_X\sigma)(Y,Z))$.

\section{Riemannian curvature of Lie groups}\label{sec:Riem-Lie-gps}

We present, for later reference and to establish notation, some general facts about left-invariant metrics on Lie groups. In particular, we refer to \cite[Proposition 3.18]{Cheeger-Ebin}, \cite[Theorem 7.30]{Besse}, or \cite[Theorem 5.3]{Arvanitoyeorgos} for more details on the formulas involving curvature.

Let $G$ be a connected Lie group, $\mathfrak{g}$ be its Lie algebra, and $\g=\langle\cdot,\cdot\rangle$ be a left-invariant metric on $G$. For all $x,y\in \mathfrak{g}$, we set ${\rm ad}_xy = [x,y]$, and denote by ${\rm ad}_x^\dagger$ the $\g$-adjoint of ${\rm ad}_x$, characterized by $\langle {\rm ad}_xy, z\rangle = \langle y, {\rm ad}_x^\dagger z\rangle$ for all $z\in \mathfrak{g}$. The Levi-Civita connection of $\g$ is given by the Koszul formula
\begin{equation}\label{eqn:LC-B}
  \nabla_xy = \frac{1}{2}[x,y] - B(x,y),\quad\mbox{where}\quad B(x,y) = \frac{1}{2} ({\rm ad}_x^\dagger y + {\rm ad}_y^\dagger x).
\end{equation}We also have that $\g$ is bi-invariant if and only if $B=0$, as a consequence of the easily established formula $-2\langle B(x,y),z\rangle = \langle x, {\rm ad}_zy\rangle + \langle {\rm ad}_zx, y\rangle$. Indeed, $\g$ is bi-invariant if and only if each ${\rm ad}_z$ is a skew-adjoint operator \cite[Lemma 7.2]{Milnor_1976}. Moreover, it directly follows from \eqref{eqn:LC-B} that, whenever $z,v\in \mathfrak{g}$ and ${\rm ad}_z$ is skew-adjoint, the relations
\begin{equation}\label{eqn:Koszul-ad-skew}
{\rm (i)}~  \nabla_zz = 0,\quad {\rm (ii)}~ \nabla_vz = -\frac{1}{2} {\rm ad}_v^\dagger z,\quad {\rm (iii)}~\nabla_zv = {\rm ad}_zv - \frac{1}{2} {\rm ad}_v^\dagger z
\end{equation}all hold. Finally, we have that
\begin{equation}\label{eqn:left-inv-sectional}
  \begin{split}
      {\rm Sec}^\g(v,w) = &-\frac{3}{4} \|[v,w]\|^2-\frac{1}{2} \langle [[v,w],w],v\rangle\\ &-\frac{1}{2} \langle [v,[v,w]],w\rangle - \langle B(v,v),B(w,w)\rangle + \|B(v,w)\|^2
  \end{split}
\end{equation}whenever $v,w\in \mathfrak{g}$ are orthonormal vectors; in particular, whenever $B=0$, \eqref{eqn:left-inv-sectional} reduces to ${\rm Sec}^\g(v,w) = \|[v,w]\|^2/4$.

\section{Left-invariant magnetic systems and sectional curvature}\label{sec:inv-mag-curv}

Let $G$ be a connected Lie group, and $\mathfrak{g}$ its Lie algebra. A left-invariant $2$-form $\sigma$ on $G$ is closed if and only if
\begin{equation}\label{closed-2-form}
  \sigma([u,v],w) + \sigma([v,w],u) + \sigma([w,u],v) = 0,\quad\mbox{for all }u,v,w\in \mathfrak{g}.
\end{equation}If $\sigma$ is exact, and $G$ is compact or has $H^2(\mathfrak{g},\R) = \{0\}$, we may choose a left-invariant $1$-form $\alpha$ on $G$ such that $\sigma = -{\rm d}\alpha$, in which case $\sigma(u,v) = \alpha([u,v])$ for all $u,v\in \mathfrak{g}$. If $\g=\langle\cdot,\cdot\rangle$ is a left-invariant metric on $G$, we will often start with a skew-adjoint operator $\Y$ on $\mathfrak{g}$, and \emph{define} $\sigma$ via \eqref{eqn:Lorentz_force} instead; it follows from \eqref{eqn:Lorentz_force} and \eqref{closed-2-form} that
\begin{equation}\label{Y-kills-[g,g]}
  \parbox{.64\textwidth}{if $[\mathfrak{g},\mathfrak{g}]\subseteq \ker \Y$, the corresponding $2$-form $\sigma$ is closed.}
\end{equation}
Next, a left-invariant covariant tensor field $\Theta$ is bi-invariant if and only if it is ${\rm Ad}$-invariant. This latter condition, as $G$ is connected, may be tested for infinitesimally: it holds if and only if\begin{equation}\label{ad-invariance}
  \sum_{i=1}^k \Theta(v_1,\ldots, v_{i-1}, {\rm ad}_zv_i, v_{i+1},\ldots, v_k) = 0,
\end{equation}for all $z,v_1,\ldots, v_k \in \mathfrak{g}$. It now follows that
\begin{equation}\label{eqn:bi-inv-iff-parallel}
  \parbox{.86\textwidth}{if the metric $\g$ is bi-invariant, $\Theta$ is bi-invariant if and only if $\Theta$ is parallel.}
\end{equation}
 Indeed, the left side of \eqref{ad-invariance} equals $-2(\nabla_z\Theta)(v_1,\ldots, v_k)$, due to \eqref{eqn:LC-B} and the definition of $\nabla_z\Theta$. Bi-invariant magnetic systems will be discussed in more detail in Section \ref{sec:non-triv-bi}, but we may already register a very useful characterization which follows from the above discussion.

\begin{prop}\label{prop:bi-inv-char}
 A left-invariant magnetic system $(\g,\sigma)$ on a connected Lie group $G$ is bi-invariant if and only if, for every $z\in \mathfrak{g}$, ${\rm ad}_z$ is skew-adjoint and commutes with $\Y$.
\end{prop}

\begin{proof}
  We know from \cite[Lemma 7.2]{Milnor_1976} that $\g$ is bi-invariant if and only if for each $z\in \mathfrak{g}$, ${\rm ad}_z$ is skew-adjoint. In this case \eqref{eqn:LC-B} yields $2\nabla_z\Y = [{\rm ad}_z,\Y]$ and, noting that $\sigma$ is parallel if and only if $\Y$ is parallel, we may invoke \eqref{eqn:bi-inv-iff-parallel}.
\end{proof}

We now proceed to compute the magnetic curvature of a left-invariant magnetic system $(\g,\sigma)$. The curvatures will always be evaluated at vectors in $\mathfrak{g}$, and so we omit the identity base point $e\in G$, using the respective shorthand notations $\Y$, $A^{(\g,\sigma)}_v$, ${\rm Sec}^{(\g,\sigma,s)}(v,w)$, and ${\rm Ric}^{(\g,\sigma,s)}(v)$ for $\Y_e$, $A^{(\g,\sigma)}_{(e,v)}$, ${\rm Sec}^{(\g,\sigma,s)}(e,v)$, and ${\rm Ric}^{(\g,\sigma,s)}(e,v)$, where $v,w\in \mathfrak{g}$. First, we claim that
\begin{equation}\label{eqn:linear-term}
  \begin{split}\langle (\nabla_w\Y)(v),w\rangle  &= \langle [w,\Y(v)], w\rangle + \frac{1}{2} \langle [w,v],\Y(w)\rangle \\ &\quad - \frac{1}{2} \langle [v,\Y(w)], w\rangle - \frac{1}{2} \langle [w,\Y(w)],v\rangle
  \end{split}
  \end{equation}
  holds, for any $v,w\in \mathfrak{g}$. Indeed, by definition of $\nabla_w \Y$ and skew-adjointness of $\Y$, we have the relation $\langle (\nabla_w\Y)(v),w\rangle = \langle \nabla_w(\Y(v)), w\rangle + \langle \nabla_wv, \Y(w)\rangle$. Using \eqref{eqn:LC-B} and that $[w,w]=0$, it follows that $\langle \nabla_w(\Y(v)),w\rangle = \langle [w,\Y(v)], w\rangle$. A similar calculation again using \eqref{eqn:LC-B} shows that the second term $\langle \nabla_wv, \Y(w)\rangle$ amounts to the remaining three terms in \eqref{eqn:linear-term}. We register next the result of plugging \eqref{eqn:Aww}, \eqref{eqn:left-inv-sectional}, and \eqref{eqn:linear-term} into \eqref{eqn:mag-sec-quadratic}.

\begin{prop}\label{prop:left-invariant-mag-sec}
  Let $(\g,\sigma)$ be any left-invariant magnetic system on $G$, and $s>0$. Then, we have that
  \begin{equation}\label{eqn:left-invariant-mag-sec}
    \begin{split}
      {\rm Sec}^{(\g,\sigma,s)}(v,w)  &= s^2\Big(-\frac{3}{4} \|[v,w]\|^2-\frac{1}{2} \langle [[v,w],w],v\rangle  -\frac{1}{2} \langle [v,[v,w]],w\rangle   \\ &\quad - \langle B(v,v),B(w,w)\rangle + \|B(v,w)\|^2\Big)-  s\Big(\langle [w,\Y(v)],w\rangle  \\ &\quad + \frac{1}{2} \langle [w,v],\Y(w)\rangle -\frac{1}{2} \langle [v,\Y(w)],w\rangle-\frac{1}{2} \langle [w,\Y(w)],v\rangle\Big) \\ &\quad +\frac{3}{4} \langle w,\Y(v)\rangle^2 + \frac{1}{4} \|\Y(w)\|^2,
    \end{split}
  \end{equation}whenever $v,w\in\mathfrak{g}$ are orthonormal, for $B$ as in \eqref{eqn:LC-B}.
\end{prop}

The next result generalizes \cite[Lemma 1.6]{Milnor_1976}, which states that the Riemannian sectional curvature of a left-invariant metric is nonnegative along planes containing directions $z$ with skew-adjoint ${\rm ad}_z$. The new assumption considered below, which means that the one-parameter subgroup generated by $z$ acts on $G$ via conjugations preserving \emph{both} $\g$ and $\sigma$, is of course motivated by Proposition \ref{prop:bi-inv-char}.

\begin{prop}\label{prop:sec-central}
  Let $z \in \mathfrak{g}$ be a unit vector such that ${\rm ad}_z$ is skew-adjoint and commutes with $\Y$. Then, ${\rm Sec}^{(\g,\sigma,s)}(v,z) \geq 0$ for every $s>0$ and unit vector $v\in \mathfrak{g}$ orthogonal to $z$, with equality if and only if $z$ is orthogonal to the image of ${\rm ad}_v - s^{-1}\Y$. In particular, this holds if $z$ is a central element.
\end{prop}

\begin{proof}
 Using that $z$ is orthogonal to the images of ${\rm ad}_z$ and $v\mapsto {\rm ad}^\dagger_v z$ whenever ${\rm ad}_z$ is skew-adjoint, a routine computation reduces \eqref{eqn:left-inv-sectional} to ${\rm Sec}^\g(v,z) = \|{\rm ad}_v^\dagger z\|^2/4$. Next, using also the condition that $[{\rm ad}_z,\Y]=0$, only the third term in \eqref{eqn:linear-term} remains, which we write as $\langle (\nabla_z\Y)(v),z\rangle = -\langle \Y(z), {\rm ad}_v^\dagger z\rangle/2$. Substituting all of it into \eqref{eqn:mag-sec-quadratic} and factoring it with \eqref{eqn:Aww}, we arrive at
  \begin{equation}
    {\rm Sec}^{(\g,\sigma,s)}(v,z) = \frac{s^2}{4} \| {\rm ad}_v^\dagger z + s^{-1}\Y(z)\|^2 + \frac{3}{4} \langle \Y(z),v\rangle^2 \geq 0,
  \end{equation}with equality if and only if $\,{\rm ad}_v^\dagger z + s^{-1}\Y(z) = 0$ and $\Y(z)$ is orthogonal to $v$. However, the latter condition is implied by the former, and the conclusion follows from the general relation $\langle {\rm ad}_v^\dagger z + s^{-1} \Y(z), w\rangle = \langle z, {\rm ad}_vw -s^{-1}\Y(w)\rangle$, valid for all $w\in \mathfrak{g}$.
\end{proof}

Theorem \ref{thm:bi-inv-pos-sec} now follows from Proposition \ref{prop:sec-central}, in view of Proposition \ref{prop:bi-inv-char} together with the easily verified claim that ${\rm ad}_v^\dagger z+s^{-1}\Y(z)$ vanishes for all $v,z\in \mathfrak{g}$ if and only if ${\rm ad}_v=0$ for all $v$ (i.e., $G$ is Abelian) and $\Y=0$ (hence $\sigma=0$). Alternatively, one may simplify \eqref{eqn:left-invariant-mag-sec} in the case where $(\g,\sigma)$ is bi-invariant to obtain
  \begin{equation}\label{eqn:mag-sec-bi-inv}
    {\rm Sec}^{(\g,\sigma,s)}(v,w) = \frac{1}{4}\left(s^2\|[v,w]\|^2  +3 \langle w,\Y(v)\rangle^2 + \|\Y(w)\|^2\right),
  \end{equation}for all orthonormal $v,w\in\mathfrak{g}$.
\begin{ex}
  The assumption that ${\rm ad}_z$ commutes with $\Y$ in Proposition \ref{prop:sec-central} cannot be dropped. Consider the Lie group $\mathrm{E}(2)$ of all rigid motions of the plane, equipped with the left-invariant metric $\g$ for which the basis $x,y,z$ of $\mathfrak{e}(2)$ satisfying the bracket relations $[z,x]=-y$, $[z,y]=x$, and $[x,y]=0$ declared to be $\g$-orthonormal. Note that every left-invariant $2$-form on $\mathrm{E}(2)$ is automatically closed: as $\mathrm{E}(2)$ is three-dimensional, it suffices to check that \eqref{closed-2-form} is satisfied when evaluated at $(u,v,w) = (x,y,z)$; this is obvious from the given bracket relations. The operator $\Y$ defined by $\Y(x)=0$, $\Y(y)=z$, $\Y(z)=-y$, is a rotation in the $yz$-plane, and thus skew-adjoint; define $\sigma$ via \eqref{eqn:Lorentz_force}. The operator ${\rm ad}_z$ is also skew-adjoint, being another rotation, but $[{\rm ad}_z,\Y](x)=z \neq 0$. Using \eqref{eqn:Koszul-ad-skew}, we compute that $\nabla_zx = -y$ and $\nabla_xz = \nabla_yz = \nabla_zz=0$, so that $R^\g(x,z)z = 0$ and $\langle (\nabla_z\Y)(x),z\rangle = 1$. From \eqref{eqn:mag-sec-quadratic}, we finally obtain that ${\rm Sec}^{(\g,\sigma,s)}(x,z) = -s+1/4 < 0$ for all $s>1/4$.
\end{ex}

\begin{ex}\label{ex:aff(C)}
  In view of \eqref{eqn:sec-not-symmetric}, one may also ask if the natural analogue of Proposition \ref{prop:sec-central} holds, with ${\rm Sec}^{(\g,\sigma,s)}(z,w) \geq 0$ for all $s>0$ and unit vector $w\in \mathfrak{g}$ orthogonal to $z$, provided that ${\rm ad}_z$ is skew-adjoint and commutes with $\Y$. This is, however, false: consider the group ${\rm Aff}(\mathbb{C})$ of affine functions on $\mathbb{C}$, equipped with the left-invariant metric $\g$ for which the basis $x,y,z,w$ of $\mathfrak{aff}(\mathbb{C})$ having the nonzero bracket relations $[x,w]=w$, $[x,y]=y$, $[z,w]=y$, and $[z,y]=-w$, is $\g$-orthonormal. Using \eqref{eqn:LC-B} and \eqref{eqn:Koszul-ad-skew}, we fully describe the Levi-Civita connection of $\g$ in Table \ref{tab:LC-aff}.
  \begin{table}[H]
    \centering
    \begingroup\renewcommand{\arraystretch}{1.2}\begin{equation}
    \begin{array}{|c||c|c|c|c|}
    \hline \nabla & x & y & z & w \\  \hline\hline x & 0 & 0 & \,\,0\,\, & \,\,0\,\,  \\ \hline y & -y & x & 0 & 0 \\ \hline z & 0 & -w & 0 & y \\ \hline w & -w & 0 & 0 & x \\ \hline
    \end{array}
  \end{equation}\endgroup \medskip
    \caption{The Levi-Civita product of $\g$ in $\mathfrak{aff}(\C)$.}
    \label{tab:LC-aff}
  \end{table}
  \noindent The operator $\Y$ defined by $\Y(x) = z$, $\Y(y)=0$, $\Y(z) = -x$, and $\Y(w) = 0$, being a rotation in the $xz$-plane, is skew-adjoint and defines a closed $2$-form $\sigma$, by \eqref{eqn:Lorentz_force} and \eqref{Y-kills-[g,g]}. It also commutes with ${\rm ad}_z$, which is a similar rotation in the $yw$-plane, and hence skew-adjoint as well. Using \eqref{eqn:mag-sec-quadratic} and Table \ref{tab:LC-aff}, we find that
  \begin{equation}\label{eqn:sec-aff(C)}
    {\rm Sec}^{(\g,\sigma,s)}(z,x) = 1, \quad\mbox{and}\quad {\rm Sec}^{(\g,\sigma,s)}(z,y) = {\rm Sec}^{(\g,\sigma,s)}(z,w)=-s,
  \end{equation}for all $s>0$, making the latter expression manifestly negative.
\end{ex}

\begin{obs}
  In \cite[Lemmas 2.1 and 2.3]{Milnor_1976}, two results on Ricci curvature are presented: (i) If $z\in \mathfrak{g}$ has skew-adjoint ${\rm ad}_z$, then ${\rm Ric}^\g(z) \geq 0$, with equality if and only if $z\in [\mathfrak{g},\mathfrak{g}]^\perp$, and (ii) if $z\in [\mathfrak{g},\mathfrak{g}]^\perp$, then ${\rm Ric}^{\g}(z) \leq 0$, with equality if and only if ${\rm ad}_z$ is skew-adjoint. Both are false in the magnetic setting. For (i), we recycle Example \ref{ex:aff(C)}: it follows from \eqref{eqn:mag-Ric} and \eqref{eqn:sec-aff(C)} that ${\rm Ric}^{(\g,\sigma,s)}(z) = 1-2s < 0$ for $s>1/2$. As for (ii), note from \eqref{eqn:Aww} that ${\rm tr}(A^{(\g,\sigma)}_{(x,v)}) \geq 0$, with equality if and only if $\Y_x=0$, so that making $s\to 0^+$ in \eqref{eqn:mag-Ric-div} shows that ${\rm Ric}^{(\g,\sigma,s)}(x,v) > 0$ for $s$ sufficiently small whenever $\Y_x\neq 0$, regardless of any properties one might require of $v$.
\end{obs}

To conclude this section, we present the proof of Theorem \ref{teo:solvable}.

\begin{proof}[Proof of Theorem \ref{teo:solvable}]
  Using \eqref{eqn:mag-sec-quadratic} twice we obtain
  \begin{equation}\label{eqn:mag-avg-sec}
    \frac{1}{2} ({\rm Sec}^{(\g,\sigma,s)}(v,w)+{\rm Sec}^{(\g,\sigma,s)}(-v,w)) = s^2{\rm Sec}^g(v,w) +  \langle A^{(\g,\sigma)}_v(w),w\rangle
  \end{equation}for any vectors $v,w\in \mathfrak{g}$ with $v$ unit and $w$ orthogonal to $v$. If ${\rm Sec}^{(\g,\sigma,s)} \leq 0$, \eqref{eqn:mag-avg-sec} and \eqref{eqn:Aww} directly lead to
  \begin{equation}\label{eqn:neg-mag-sec}
    {\rm Sec}^\g(v,w) \leq -\frac{1}{4s^2} (3\langle \Y(v),w\rangle^2+\|\Y(w)\|^2) \leq 0,
  \end{equation}allowing us to apply \cite[Theorem 1.6]{Milnor_1976} and conclude that $G$ is solvable. If $G$ is also unimodular, then $\g=0$, and thus \eqref{eqn:neg-mag-sec} yields $3\langle \Y(v),w\rangle^2+\|\Y(w)\|^2 = 0$, which is only possible when $\Y$ (and hence $\sigma$) vanishes identically.
\end{proof}

\section{More on the bi-invariant case}\label{sec:non-triv-bi}

It is well-known that a Lie group admits a bi-invariant metric if and only if $G$ is isomorphic to the product of a compact group and an Abelian group \cite[p. 297]{Milnor_1976}. We may clearly replace ``bi-invariant metric'' with ``bi-invariant magnetic system'' here, by simply considering systems of the form $(\g,0)$. This raises the question of when is it possible to find a \emph{nontrivial} bi-invariant magnetic system, that is, one with $\sigma \neq 0$.

For example, rearranging \eqref{ad-invariance}, we have that
\begin{equation}\label{eqn:bilinear-bi-invariant}
\parbox{.61\textwidth}{a left-invariant bilinear form $\Theta$ is bi-invariant if and only if $\,\Theta([x,y],z) = \Theta(x,[y,z])\,$ for all $\,x,y,z\in \mathfrak{g}$.}  
\end{equation}
 With this, we claim that
\begin{equation}\label{eqn:perfect-bi-inv}
  \parbox{.56\textwidth}{if $\mathfrak{g} = [\mathfrak{g},\mathfrak{g}]$ (that is, $\mathfrak{g}$ is \emph{perfect}), the only bi-invariant bilinear form on $\mathfrak{g}$ is the trivial one.}
\end{equation}
 Indeed, if $\Theta$ is bi-invariant, the trilinear form $(x,y,z) \mapsto \Theta([x,y],z)$ is at the same time skew-symmetric in the first pair of arguments, and symmetric in the second pair by \eqref{eqn:bilinear-bi-invariant}, and hence vanishes, implying that $\Theta([\mathfrak{g},\mathfrak{g}],\mathfrak{g}) = 0$, and thus $\Theta=0$ as required. A similar argument shows that, in general, every bi-invariant $2$-form is automatically closed: applying bi-invariance of $\sigma$ to each term in the left side of \eqref{closed-2-form} yields ${\rm d}\sigma = -{\rm d}\sigma$.

Below, recall that a Lie algebra $\mathfrak{g}$ is called \emph{reductive} if there is a direct-sum decomposition
\begin{equation}\label{eqn:reductive-defn}
  \mathfrak{g} = \mathfrak{z(g)} \oplus \bigoplus_{i=1}^k \mathfrak{g}_i,\quad\mbox{with each of }\mathfrak{g}_1,\ldots, \mathfrak{g}_k\mbox{ simple,}
\end{equation}where $\mathfrak{z}(\mathfrak{g})$ is the center of $\mathfrak{g}$. This is a natural condition to consider, being automatically satisfied whenever $G$ admits a bi-invariant metric. Namely, in this case \eqref{eqn:bilinear-bi-invariant} ensures that $\mathfrak{z(g)}= [\mathfrak{g},\mathfrak{g}]^\perp$, yielding
\begin{equation}\label{eqn:z-plus-c}
  \parbox{.62\textwidth}{the orthogonal direct-sum splitting $\mathfrak{g} = \mathfrak{z}(\mathfrak{g}) \oplus [\mathfrak{g},\mathfrak{g}]$,}
\end{equation}
 with $[\mathfrak{g},\mathfrak{g}]$ semisimple: by \eqref{eqn:z-plus-c}, any Abelian ideal $\mathfrak{a}\subseteq [\mathfrak{g},\mathfrak{g}]$ must in fact commute with all of $\mathfrak{g}$, leading to $\mathfrak{a} \subseteq \mathfrak{z}(\mathfrak{g})\cap [\mathfrak{g},\mathfrak{g}]= \{0\}$.

\begin{prop}\label{prop:restriction-isomorphism}
  If $\mathfrak{g}$ is reductive, the restriction mapping
  \begin{equation}\label{eqn:restriction-isomorphism}
\left(\scalebox{1.2}{$\wedge$}^2 \mathfrak{g}^*\right)^G \ni \sigma \mapsto \sigma|_{\mathfrak{z(g)}\times \mathfrak{z(g)}} \in \scalebox{1.2}{$\wedge$}^2\mathfrak{z(g)}^*
  \end{equation}is an isomorphism; here, $\left(\scalebox{1.2}{$\wedge$}^2 \mathfrak{g}^*\right)^G$ denotes the space of bi-invariant $2$-forms on $\mathfrak{g}$. 
\end{prop}

\begin{proof}
  For a decomposition of $\mathfrak{g}$ as in \eqref{eqn:reductive-defn}, and any bi-invariant $2$-form $\sigma$ on $G$, we first claim that
  \begin{equation}\label{sigma-gi-gj}
    \parbox{.73\textwidth}{\textrm{(i)} $\sigma(\mathfrak{g}_i,\mathfrak{g}_j) = 0$, and \textrm{(ii)} $\sigma(\mathfrak{g}_i,\mathfrak{z(g)}) = 0$, for all $i,j=1,\ldots, k$.}
  \end{equation}
  If $i=j$, \mbox{(\ref{sigma-gi-gj}-i)} follows from \eqref{eqn:perfect-bi-inv} as each $\mathfrak{g}_i$, being simple, is also perfect. Assuming that $i\neq j$ instead, let $\mathfrak{b}$ be either $\mathfrak{g}_j$ or $\mathfrak{z(g)}$, take $x,y\in \mathfrak{g}_i$ and $z\in \mathfrak{b}$, and note that $\sigma([x,y],z) = \sigma(x,[y,z]) = 0$ as $[\mathfrak{g}_i,\mathfrak{b}]=0$, so that $\sigma([\mathfrak{g}_i,\mathfrak{g}_i],\mathfrak{b})=0$, leading to both \mbox{(\ref{sigma-gi-gj}-i)} and \mbox{(\ref{sigma-gi-gj}-ii)} due to perfectness of $\mathfrak{g}_i$. Finally, condition \eqref{sigma-gi-gj} amounts to injectivity of \eqref{eqn:restriction-isomorphism}, while its surjectivity is clear: given a $2$-form on $\mathfrak{z(g)}$, extend it by zero via \eqref{eqn:reductive-defn} and \eqref{sigma-gi-gj} to a $2$-form on $\mathfrak{g}$; bi-invariance of the resulting extension is immediate from \eqref{eqn:bilinear-bi-invariant}, as all relevant brackets vanish.
\end{proof}

We may now establish the main result of this section.

\begin{teo}
  A connected Lie group $G$ admits a \emph{non-trivial} bi-invariant magnetic system if and only if $G$ is isomorphic to a direct product $K\times \R^m$, where $\R^m$ is Abelian and $K$ is a compact Lie group with $\dim \mathfrak{z(k)} + m \geq 2$.
\end{teo}

\begin{proof}
 It follows from Proposition \ref{prop:restriction-isomorphism} that, in the presence of a bi-invariant metric, nonzero bi-invariant $2$-forms exist if and only if $\dim \mathfrak{z(g)} \geq 2$; this is precisely the condition needed for $\dim \left(\scalebox{1.2}{$\wedge$}^2 \mathfrak{g}^*\right)^G > 0$. We now invoke \cite[p. 297]{Milnor_1976}, stated in the beginning of this section, and note that $\dim \mathfrak{z(g)} = \dim \mathfrak{z(k)} + m$.
\end{proof}

\begin{cor}
  A connected and semisimple Lie group does not admit nontrivial bi-invariant magnetic systems.
\end{cor}

\begin{ex}
  If $\g$ is a bi-invariant metric on $G$, we may consider $\Y = {\rm ad}_\xi$ for some $\xi \in \mathfrak{g}$; the corresponding 2-form $\sigma_\xi$ obtained via \eqref{eqn:Lorentz_force} is closed: \eqref{closed-2-form} becomes the Jacobi identity. This example, at first natural, has a crucial ``flaw'': if $(\g,\sigma_\xi)$ is bi-invariant, then $\sigma_\xi=0$. Indeed, in this case Proposition \ref{prop:bi-inv-char} tells us that $0= [{\rm ad}_\xi, {\rm ad}_x] = {\rm ad}_{[\xi,x]}$ for all $x\in \mathfrak{g}$, so that $[\xi,x] \in \mathfrak{z(g)} \cap [\mathfrak{g},\mathfrak{g}]=\{0\}$, cf. \eqref{eqn:z-plus-c}, and thus $\Y={\rm ad}_\xi=0$. When $\g$ is bi-invariant and $\sigma_\xi$ is not, from \eqref{eqn:mag-sec-quadratic} and \eqref{eqn:mag-Ric} we still obtain
  \begin{equation}
    \begin{split}
    {\rm Sec}^{(\g,\sigma_\xi, s)}(v,w) &= \frac{1}{4} \|[sv-\xi, w]\|^2 +\frac{3}{4}\langle [\xi,v],w\rangle^2       \\ \mbox{and}\quad {\rm Ric}^{(\g,\sigma_\xi, s)}(v) &= -\frac{1}{4} \beta(sv-\xi,sv-\xi) + \frac{3}{4} \|[\xi,v]\|^2
    \end{split}
  \end{equation}for all orthonormal $v,w\in\mathfrak{g}$, where $\beta$ is the Killing form of $\mathfrak{g}$.
\end{ex}

\section{Magnetic flatness of left-invariant systems}\label{sec:left-inv-flatness}

The meaning of magnetic flatness for a general magnetic system $(\g,\sigma)$ on a smooth manifold $M$ is well-understood:
\begin{equation}\label{mag-flatness}
  \parbox{.79\textwidth}{If $s>0$ is such that ${\rm Sec}^{(\g,\sigma,s)} = 0$, then $\nabla\sigma=0$ and one of the following must hold: (i) $\sigma=0$ and $\g$ is flat, or (ii) $\sigma$ is nowhere-vanishing and $J = \|\Y\|^{-1}\Y$ turns $(M,\g)$ into a K\"{a}hler manifold of constant negative holomorphic sectional curvature $-\|\Y\|^2/s^2$,} 
\end{equation}cf. \cite[Theorem D]{ART_2024}. Here, $\|\Y\|$ denotes the operator norm of $\Y$. The goal of this section is to refine \eqref{mag-flatness} in the case of left-invariant magnetic systems on Lie groups, by proving Theorem \ref{teo:Milnor-Flat}. One general observation we will need is that if a K\"{a}hler manifold has constant holomorphic sectional curvature $k$, then its sectional curvature is bounded between $k/4$ and $k$, cf. \cite[Propositions 7.3 and 7.4]{KN2}; in particular, it follows that if $k < 0$, both the sectional and Ricci curvatures are also strictly negative.

\begin{proof}[Proof of Theorem \ref{teo:Milnor-Flat}]
  If $\sigma=0$, this amounts to \cite[Theorem 1.5]{Milnor_1976}. If $\sigma$ is nowhere-vanishing instead, it follows from \mbox{(\ref{mag-flatness}-ii)} and the above observation that $\g$ is negatively curved, allowing us to invoke the result by Heintze mentioned in the Introduction---it follows that $(G,\g)$ with the complex structure $J = \|\Y\|^{-1}\Y$ is holomorphically isometric to a $\C\mathbb{H}^n$ \cite[Theorem 4]{Heintze_1974}. Conversely, if $(G,\g)$ equipped with $J=\|\Y\|^{-1}\Y$ is holomorphically isometric to a $\C\mathbb{H}^n$ with constant holomorphic curvature $k<0$, then $(\g,\sigma)$ is in particular a K\"{a}hler magnetic system (as $\|\Y\|$ is constant by left-invariance of $\g$ and $\Y$), with $s$-magnetic sectional curvature given by
  \begin{equation}
   {\rm Sec}^{(\g,\sigma,s)}(v,w) = \frac{s^2k+\|\Y\|^2}{4}(1+3\langle v,Jw\rangle^2),
  \end{equation}cf. \cite[Formula (6.2)]{ART_2024}; it vanishes identically for $s = \|\Y\|/\sqrt{-k}$.
\end{proof}

\section{The Heisenberg group}\label{sec:Heisenberg}

Let $(V,\Omega)$ be a symplectic vector space, and denote by ${\rm H}(V,\Omega)$ its abstract Heisenberg group. Namely, ${\rm H}(V,\Omega)$ is the Cartesian product $V\times \R$ equipped with the group operation $(x,a)(y,b) = (x+y,a+b + \Omega(x,y)/2)$, and its Lie algebra $\mathfrak{h}(V,\Omega)=V\times \R$ has the bracket $[(v,a),(w,b)] = (0,\Omega(v,w))$.  The group ${\rm H}(V,\Omega)$ is two-step nilpotent, its center equals $\{0\}\times \R$ due to nondegeneracy of $\Omega$, and so the resulting short exact sequence $0 \to \R \to {\rm H}(V,\Omega) \to V \to 0$ yields that ${\rm H}(V,\Omega)$ is a central extension of $V$ by $\R$. Let $\langle\cdot,\cdot\rangle_0$ be any inner product on $V$, and consider the left-invariant magnetic system $(\g,\sigma)$ on ${\rm H}(V,\Omega)$ given by
\begin{equation}\label{eqn:sys-heisenberg}
  \g((v,a),(w,b)) = \langle v,w\rangle_0+ab\quad\mbox{and}\quad \sigma((v,a),(w,b)) = \Omega(v,w)
\end{equation}on $\mathfrak{h}(V,\Omega)$. From \eqref{Y-kills-[g,g]}, $\sigma$ is closed; in fact, $\sigma = - {\rm d}\alpha$, where $\alpha\colon \mathfrak{h}(V,\Omega)\to \R$ denotes the second coordinate projection. Note that all left-invariant magnetic systems on ${\rm H}(V,\Omega)$ having a left-invariant potential 1-form and making $V$ and $\R$  orthogonal are, up to rescaling, of the form \eqref{eqn:sys-heisenberg}. The Lorentz force operator of $(\g,\sigma)$ is given by $\Y(v,a) = (\Y_0(v),0)$, where $\Y_0$ is the operator on $V$ characterized $\langle \Y_0(v),w\rangle_0 = \Omega(v,w)$ for all $v,w\in V$. We similarly have that
\begin{equation}\label{eqn:ad-heisenberg}
  {\rm ad}_{(v,a)}(w,b) = (0, \Omega(v,w))\quad\mbox{and}\quad {\rm ad}_{(v,a)}^\dagger(w,b) = (b \Y_0(v), 0),
\end{equation}and hence the only elements $(v,a)$ for which ${\rm ad}_{(v,a)}$ is skew-adjoint are precisely the ones with $v=0$. Setting ${\bf z} = (0,1)$ and identifying $(v,a)$ with $v+a{\bf z}$, we use \eqref{eqn:LC-B}, \eqref{eqn:Koszul-ad-skew}, and \eqref{eqn:ad-heisenberg} to obtain
\begin{equation}\label{eqn:LC-heisenberg}
  \nabla_vw = \frac{1}{2}\Omega(v,w){\bf z},\quad  \nabla_{\bf z}v = \nabla_v{\bf z} = -\frac{1}{2}\Y_0(v),\quad\mbox{and}\quad\nabla_{\bf z}{\bf z} = 0.
\end{equation}
Noting that ${\bf z}$ is central with $\Y({\bf z})=0$, and $\g([u,v],w)=0$ for all $u,v,w\in V$, we use \eqref{eqn:LC-heisenberg} to evaluate the three coefficients in \eqref{eqn:mag-sec-quadratic}, so that
\begin{equation}\label{eqn:mag-sec-heisenberg}
 \begin{split} {\rm Sec}^{(\g,\sigma,s)}(v+a{\bf z}, w+b{\bf z}) &=  s^2\left(\frac{a^2}{4} \|\Y_0(w)\|^2+\frac{b^2}{4} \|\Y_0(v)\|^2 \right. \\ &\quad\left. -\frac{3}{4}\Omega(v,w)^2-\frac{ab}{2} \langle \Y_0(v),\Y_0(w)\rangle\right) \\ &\quad+ s\left(\frac{a}{2} \|\Y_0(w)\|^2-\frac{b}{2}\langle \Y_0(v),\Y_0(w)\rangle\right) \\ &\quad+\frac{3}{4}\Omega(v,w)^2+\frac{1}{4}\|\Y_0(w)\|^2\end{split}
\end{equation}for all orthonormal $v+a{\bf z},w+b{\bf z}\in \mathfrak{h}(V,\Omega)$. Notable particular instances of \eqref{eqn:mag-sec-heisenberg}, for orthonormal $v,w\in V$, are
\begin{equation}\label{eqn:specific-sec-heisenberg}
  \begin{split}
     {\rm (i)}&~ {\rm Sec}^{(\g,\sigma,s)}(v,w)  = \frac{1}{4}\left( 3(1-s^2) \Omega(v,w)^2  + \|\Y_0(w)\|^2\right), \\ {\rm (ii)}&~{\rm Sec}^{(\g,\sigma,s)}(v, \pm {\bf z}) = \frac{s^2}{4} \|\Y_0(v)\|^2,\quad\mbox{and}  \\ {\rm (iii)}&~{\rm Sec}^{(\g,\sigma,s)}(\pm{\bf z}, w) = \frac{(s \pm 1)^2}{4} \|\Y_0(w)\|^2.
  \end{split}
\end{equation}We compute the $s$-magnetic Ricci curvature of $(\g,\sigma)$ as well, using relation \eqref{eqn:mag-Ric}. Setting $\|\Y_0\| = \sqrt{-{\rm tr}(\Y_0^2)}>0$, it is clear from \mbox{(\ref{eqn:specific-sec-heisenberg}-iii)} that
\begin{equation}\label{eqn:Ric-pmz}
  {\rm Ric}^{(\g,\sigma,s)}(\pm{\bf z}) = \frac{(s\pm 1)^2}{4} \|\Y_0\|^2,
\end{equation}and it remains to find ${\rm Ric}^{(\g,\sigma,s)}(v+a{\bf z})$, for a unit vector $v+a{\bf z} \in \mathfrak{h}(V,\Omega)$ with $v\neq 0$ (which itself does \emph{not} have unit length). One orthonormal basis for the orthogonal complement of $v+a{\bf z}$ in $\mathfrak{h}(V,\Omega)$ is $e_1,\ldots, e_{n-1}, -a\|v\|_0^{-1} v + \|v\|_0{\bf z}$, where $e_1,\ldots, e_{n-1}$ is an orthonormal basis for the orthogonal complement of $v$ in $V$, cf. Figure \ref{fig:heis-orth}.
\begin{figure}[H]
  \centering
  \tikzset{every picture/.style={line width=0.75pt}} 
\begin{tikzpicture}[x=0.75pt,y=0.75pt,yscale=-1,xscale=1]
\draw   (199.7,103) -- (493,103) -- (367.3,190.5) -- (74,190.5) -- cycle ;
\draw    (296.5,140) -- (296.01,25) ;
\draw [shift={(296,23)}, rotate = 89.76] [color={rgb, 255:red, 0; green, 0; blue, 0 }  ][line width=0.75]    (10.93,-3.29) .. controls (6.95,-1.4) and (3.31,-0.3) .. (0,0) .. controls (3.31,0.3) and (6.95,1.4) .. (10.93,3.29)   ;
\draw  [dash pattern={on 4.5pt off 4.5pt}]  (296.5,140) -- (296.5,177.5) ;
\draw    (296.5,177.5) -- (296,239.5) ;
\draw    (296.5,140) -- (340.11,86.33) ;
\draw [shift={(342,84)}, rotate = 129.09] [fill={rgb, 255:red, 0; green, 0; blue, 0 }  ][line width=0.08]  [draw opacity=0] (8.93,-4.29) -- (0,0) -- (8.93,4.29) -- (5.93,0) -- cycle    ;
\draw  [dash pattern={on 0.84pt off 2.51pt}]  (342,84) -- (342,170) ;
\draw    (296.5,140) -- (339.5,168.35) ;
\draw [shift={(342,170)}, rotate = 213.4] [fill={rgb, 255:red, 0; green, 0; blue, 0 }  ][line width=0.08]  [draw opacity=0] (8.93,-4.29) -- (0,0) -- (8.93,4.29) -- (5.93,0) -- cycle    ;
\draw    (296.5,140) -- (252.69,180.95) ;
\draw [shift={(250.5,183)}, rotate = 316.93] [fill={rgb, 255:red, 0; green, 0; blue, 0 }  ][line width=0.08]  [draw opacity=0] (10.72,-5.15) -- (0,0) -- (10.72,5.15) -- (7.12,0) -- cycle    ;
\draw  [draw opacity=0][dash pattern={on 4.5pt off 4.5pt}] (251.54,88.04) .. controls (274.56,65.37) and (311.84,62.66) .. (339.11,83.08) .. controls (340.13,83.84) and (341.12,84.62) .. (342.09,85.43) -- (299.43,136.09) -- cycle ; \draw  [dash pattern={on 4.5pt off 4.5pt}] (251.54,88.04) .. controls (274.56,65.37) and (311.84,62.66) .. (339.11,83.08) .. controls (340.13,83.84) and (341.12,84.62) .. (342.09,85.43) ;  
\draw    (296.5,140) -- (253.51,90.31) ;
\draw [shift={(251.54,88.04)}, rotate = 49.13] [fill={rgb, 255:red, 0; green, 0; blue, 0 }  ][line width=0.08]  [draw opacity=0] (8.93,-4.29) -- (0,0) -- (8.93,4.29) -- (5.93,0) -- cycle    ;

\draw (302.5,19.9) node [anchor=north west][inner sep=0.75pt]    {$\mathbb{R}$};
\draw (475,85) node [anchor=north west][inner sep=0.75pt]    {$V$};
\draw (345,80) node [anchor=north west][inner sep=0.75pt]    {$v+a{\bf z}$};
\draw (345,162) node [anchor=north west][inner sep=0.75pt]    {$v$};
\draw (233,172) node [anchor=north west][inner sep=0.75pt]    {$e_{i}$};
\draw (142,80) node [anchor=north west][inner sep=0.75pt]    {$-a\| v\| ^{-1} v+\| v\| {\bf z}$};
\end{tikzpicture}
  \caption{An orthonormal basis for $(v+a{\bf z})^\perp \subseteq \mathfrak{h}(V,\Omega)$.}
  \label{fig:heis-orth}
\end{figure}
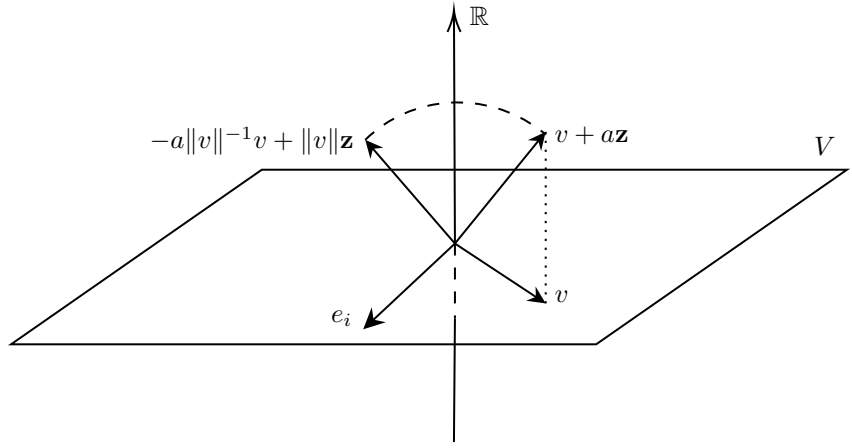
The unit-length condition $\|v\|^2 + a^2 = 1$ allows for significant simplifications in the expressions to come. Firstly, \eqref{eqn:mag-sec-heisenberg} yields
\begin{equation}\label{eqn:sec_heis_rotated}
  {\rm Sec}^{(\g,\sigma,s)}\left(v+a{\bf z}\,,\, -a\|v\|^{-1}v + \|v\|{\bf z}\right) = \frac{(s+a)^2}{4} \frac{\|\Y_0(v)\|^2}{\|v\|^2},
\end{equation}while
\begin{equation}
  \sum_{i=1}^{n-1} \|\Y_0(e_i)\|^2 = \|\Y_0\|^2 - \frac{\|\Y_0(v)\|^2}{\|v\|^2}\quad\mbox{and}\quad \sum_{i=1}^{n-1} \Omega(v,e_i)^2 = \|\Y_0(v)\|^2
\end{equation}now imply that
\begin{equation}\label{eqn:sec_heis_sum}
  \begin{split}
    \sum_{i=1}^{n-1} {\rm Sec}^{(\g,\sigma, s)}(v+a{\bf z}, e_i) &= \frac{(1+sa)^2}{4} \left(\|\Y_0\|^2 -\frac{\|\Y_0(v)\|^2}{\|v\|^2}\right) \\ &\quad + \frac{3(1-s^2)}{4} \|\Y_0(v)\|^2.
  \end{split}
\end{equation}Adding \eqref{eqn:sec_heis_rotated} to \eqref{eqn:sec_heis_sum} and factoring further leads to the surprisingly concise formula
\begin{equation}\label{eqn:ric-simplified}
  {\rm Ric}^{(\g,\sigma,s)}(v+a{\bf z}) = \frac{(1+sa)^2}{4} \|\Y_0\|^2 + \frac{(1-s^2)}{2} \|\Y_0(v)\|^2.
\end{equation}
It follows from \eqref{eqn:Ric-pmz} and \eqref{eqn:ric-simplified} that ${\rm Ric}^{(\g,\sigma,s)}>0$ whenever $0<s<1$, while ${\rm Ric}^{(\g,\sigma,s_0)}(-{\bf z}) = 0$ for $s_0=1$ by \eqref{eqn:Ric-pmz}, correctly hinting that the Ma\~{n}\'{e} critical value---which marks the most drastic changes in the geometric and dynamical behavior of the corresponding magnetic flow---is ${\bf c}(\g,\sigma)= s_0^2/2 = 1/2$, cf. \cite[Theorem 2.4]{ABM_2025} and \cite[Theorem 1]{EGM_2021}.

\appendix

\section{A magnetic Euler-Arnold equation for exact systems}\label{app:EA}

It is well-known \cite{Arnold_1966, Arnold_1989} that if $G$ is a Lie group and $\g$ is a left-invariant Riemannian metric on $G$, then a curve $\gamma\colon I\to G$ is a geodesic if and only if $x(t) = \gamma(t)^{-1}\dot{\gamma}(t)$ satisfies the so-called \emph{Euler-Arnold equation} $\dot{x}(t) = {\rm ad}_{x(t)}^\dagger x(t)$. For completeness of our presentation, we exhibit a magnetic generalization of this formalism in Proposition \ref{prop:mag-EA}; it has already been considered, for instance, in \cite{Ovando-Subils_2023, Maier_2026}.

We follow the strategy adopted in \cite[Proposition 2.1]{brigant2025conjugatepointsliegroups}. The first ingredient needed is what they call the ``zero-curvature'' equation, and a reference likely to \cite[Formula (5.1)]{Milnor_1976}. We instead provide a direct proof. 

\begin{lem}[The zero-curvature equation]
  Given a smooth curve $\gamma\colon I\to G$ and a variation $\eta\colon I\times (-\varepsilon,\varepsilon)\to G$ of $\gamma$, the translated partial derivatives
  \begin{equation}\label{eqn:translated-uv}
    u(t,s) = \eta(t,s)^{-1} \partial_t\eta(t,s)\quad\mbox{and}\quad v(t,s) = \eta(t,s)^{-1} \partial_s\eta(t,s)
  \end{equation}satisfy the equation $\partial_su - \partial_tv = {\rm ad}_uv$.
\end{lem}

\begin{proof}
  Consider the Maurer-Cartan form $\Theta\in \Omega^1(G,\mathfrak{g})$ of $G$, defined by $\Theta_g(v) =g^{-1}v$, and its well-known structure equation
  \begin{equation}\label{eq:MC-structure}
    {\rm d}\Theta(X,Y) + [\Theta(X),\Theta(Y)] = 0,\quad\mbox{for all }X,Y\in\mathfrak{X}(G),
  \end{equation}cf. \cite[Proposition 7.3]{Helgason}, which corresponds to flatness of $\Theta$ as a connection $1$-form on the principal $G$-bundle $G\to \{e\}$. Fix also any torsionfree connection $\nabla$ on $G$, so that the formula
  \begin{equation}\label{eq:d-torsionfree}
    {\rm d}\Theta(X,Y) = (\nabla_X\Theta)(Y)-(\nabla_Y\Theta)(X)
  \end{equation}holds. Finally, noting that $u = \Theta(\partial_t\eta)$ and $v = \Theta(\partial_s\eta)$, we compute:
  \begin{equation}\label{eqn:zero-curvature-computation}
    \begin{split}
      \partial_su - \partial_tv &= \partial_s (\Theta(\partial_t\eta)) - \partial_t(\Theta(\partial_s\eta)) \\ &= (\nabla_{\partial_s\eta}\Theta)(\partial_t\eta) + \Theta(\nabla_{\partial_s\eta}\partial_t\eta) -(\nabla_{\partial_t\eta}\Theta)(\partial_s\eta) - \Theta(\nabla_{\partial_t\eta}\partial_s\eta) \\ &= {\rm d}\Theta(\partial_s\eta,\partial_t\eta) + \Theta(\nabla_{\partial_s\eta}\partial_t\eta - \nabla_{\partial_t\eta}\partial_s\eta) \\ &= -[\Theta(\partial_s\eta),\Theta(\partial_t\eta)] + 0 \\ &= -[v,u] \\ &= {\rm ad}_uv, 
    \end{split}
  \end{equation}as required; in the third and fourth lines of \eqref{eqn:zero-curvature-computation} we use \eqref{eq:MC-structure} and \eqref{eq:d-torsionfree}, respectively, as well as the equality $\nabla_{\partial_s\eta}\partial_t\eta=\nabla_{\partial_t\eta}\partial_s\eta$, valid as $\nabla$ is torsionfree.
\end{proof}

The second ingredient is a variational characterization of magnetic geodesics for exact magnetic systems. Namely, if $(\g,-{\rm d}\alpha)$ is an exact magnetic system on a smooth manifold $M$, then a smooth curve $\gamma\colon [a,b]\to M$ is a $(\g,-{\rm d}\alpha)$-geodesic if and only if it is a critical point (under variations with fixed endpoints) of the \emph{Landau-Hall functional}
\begin{equation}\label{eqn:Landau-functional}
  E(\gamma) = \int_a^b\left( \frac{1}{2}\|\dot{\gamma}(t)\|^2 + \alpha_{\gamma(t)}(\dot{\gamma}(t))\right)\,{\rm d}t,
\end{equation}cf. \cite[Section I]{BRCF_2005} (specifically, the Euler-Lagrange equation of \eqref{eqn:Landau-functional} is precisely \eqref{eqn:Landau-Hall}). 

Below, suppose that $(\g,-{\rm d}\alpha)$ is a left-invariant magnetic system on the Lie group $G$, with the potential $1$-form $\alpha$ also left-invariant.

\begin{prop}\label{prop:mag-EA}
  A smooth curve $\gamma\colon [a,b] \to G$ is a $(\g,-{\rm d}\alpha)$-geodesic if and only if $x(t) = \gamma(t)^{-1}\dot{\gamma}(t)$ satisfies the \emph{magnetic Euler-Arnold equation}
  \begin{equation}\label{eqn:mag-EA}
    \dot{x}(t) = {\rm ad}_{x(t)}^\dagger(x(t) + \alpha^\sharp),
  \end{equation}
  where $\alpha^\sharp \in \mathfrak{g}$ is the vector $\g$-dual to $\alpha$. In particular, $(\g,-{\rm d}\alpha)$-geodesics are in one-to-one correspondence with integral curves of the \emph{magnetic Euler-Arnold vector field} $Z\in \mathfrak{X(g)}$, defined by $Z_x = {\rm ad}_x^\dagger(x+\alpha^\sharp)$ for all $x\in\mathfrak{g}$.
\end{prop}

\begin{proof}
  Let $\eta$ be a variation of $\gamma$ with fixed endpoints, and consider the translated partial derivatives $u$ and $v$ of $\eta$ as in \eqref{eqn:translated-uv}. By left-invariance of $\g$ and $\alpha$, we may rewrite \eqref{eqn:Landau-functional} as
  \begin{equation}\label{eqn:LH-at-identity}
    E(\eta(\cdot,s)) = \int_a^b \left(\frac{1}{2}\|u(t,s)\|^2 + \alpha_e(u(t,s))\right)\,{\rm d}t.
  \end{equation}Differentiating \eqref{eqn:LH-at-identity} with respect to $s$ (omitting arguments and setting $\g=\langle\cdot,\cdot\rangle$), we obtain
\begin{equation}
  \begin{split}
    \partial_s E(\eta) &= \int_a^b \langle \partial_su, u\rangle + \alpha(\partial_su)\,{\rm d}t \\ &\stackrel{(\ast)}{=} \int_a^b \langle {\rm ad}_uv + \partial_tv, u\rangle + \alpha({\rm ad}_uv + \partial_tv)\,{\rm d}t \\ &\stackrel{(\ast\ast)}{=} \int_a^b \langle v, {\rm ad}_u^\dagger u - \partial_tu\rangle + \langle \alpha^\sharp, {\rm ad}_uv\rangle \,{\rm d}t \\ &= \int_a^b \langle v, {\rm ad}_u^\dagger u - \partial_tu + {\rm ad}_u^\dagger (\alpha^\sharp)\rangle \,{\rm d}t.
  \end{split}
\end{equation}In $(\ast)$ we have used the zero-curvature equation, and in $(\ast\ast)$ integration by parts (noting that all boundary terms vanish as $\eta$ has fixed endpoints) together with $\alpha = \langle \alpha^\sharp,\cdot\rangle$. Thus $\gamma$ is a critical point of $E$ if and only if ${\rm ad}_x^\dagger x - \dot{x} + {\rm ad}_x^\dagger(\alpha^\sharp) = 0$, which is evidently equivalent to \eqref{eqn:mag-EA}.
\end{proof}

\begin{obs}
 The magnetic Euler-Arnold vector field does not depend on the choice of $\alpha$, in the following sense: if $\beta$ is another left-invariant $1$-form such that ${\rm d}\alpha = {\rm d}\beta$, then ${\rm ad}_x^\dagger(x+\alpha^\sharp) = {\rm ad}_x^\dagger(x+\beta^\sharp)$ for all $x\in \mathfrak{g}$. Indeed, ${\rm d}(\alpha-\beta) = 0$ means that $(\alpha-\beta)([x,y]) = 0$, and so $\langle {\rm ad}_x^\dagger(\alpha^\sharp-\beta^\sharp), y\rangle = \langle \alpha^\sharp - \beta^\sharp, [x,y]\rangle =0$ for all $x,y\in \mathfrak{g}$. Hence ${\rm ad}_x^\dagger(\alpha^\sharp-\beta^\sharp) = 0$, as claimed.
\end{obs}

\begin{obs}
  Every left-invariant magnetic system is complete. Indeed, $(\g,\sigma)$ is complete whenever $\g$ is complete \cite[Corollary 2.4]{BRCF_2005} and, more generally, every homogeneous Riemannian metric is complete \cite[Corollary 4.4]{KN1}.
\end{obs}

\bibliography{magnetic_refs}{}

\begin{thebibliography}{10}

\bibitem{AMMP_2017}
A.~Abbondandolo, L.~Macarini, M.~Mazzucchelli, and G.~P. Paternain.
\newblock Infinitely many periodic orbits of exact magnetic flows on surfaces
  for almost every subcritical energy level.
\newblock {\em J. Eur. Math. Soc. (JEMS)}, 19(2):551--579, 2017.

\bibitem{ABM_2025}
P.~Albers, G.~Benedetti, and L.~Maier.
\newblock The {H}opf-{R}inow theorem and the {M}a\~n\'e{} critical value for
  magnetic geodesics on odd-dimensional spheres.
\newblock {\em J. Geom. Phys.}, 214:Paper No. 105521, 2025.

\bibitem{Arnold_1966}
V.~I. Arnol'd.
\newblock Sur la g\'eom\'etrie diff\'erentielle des groupes de {L}ie de
  dimension infinie et ses applications \`a{} l'hydrodynamique des fluides
  parfaits.
\newblock {\em Ann. Inst. Fourier (Grenoble)}, 16:319--361, 1966.

\bibitem{Arnold_1989}
V.~I. Arnol'd.
\newblock {\em Mathematical methods of classical mechanics}, volume~60 of {\em
  Graduate Texts in Mathematics}.
\newblock Springer-Verlag, New York, second edition, 1989.
\newblock Translated from the Russian by K. Vogtmann and A. Weinstein.

\bibitem{Arvanitoyeorgos}
A.~Arvanitoyeorgos.
\newblock {\em An introduction to {L}ie groups and the geometry of homogeneous
  spaces}, volume~22 of {\em Student Mathematical Library}.
\newblock American Mathematical Society, Providence, RI, 2003.
\newblock Translated from the 1999 Greek original and revised by the author.

\bibitem{AB_2021}
L.~Asselle and G.~Benedetti.
\newblock Integrable magnetic flows on the two-torus: {Z}oll examples and
  systolic inequalities.
\newblock {\em J. Geom. Anal.}, 31(3):2924--2940, 2021.

\bibitem{Assenza_2024}
V.~Assenza.
\newblock Magnetic curvature and existence of a closed magnetic geodesic on low
  energy levels.
\newblock {\em Int. Math. Res. Not. IMRN}, (21):13586--13610, 2024.

\bibitem{mmls}
V.~Assenza, J.~De~Simoi, J.~Marshall~Reber, and I.~Terek.
\newblock Marked length spectrum rigidity for {A}nosov magnetic surfaces.
\newblock {\em Adv. Math.}, 500:Paper No. 111091, 2026.

\bibitem{ART_2024}
V.~Assenza, J.~Marshall~Reber, and I.~Terek.
\newblock Magnetic flatness and {E}. {H}opf's theorem for magnetic systems.
\newblock {\em Comm. Math. Phys.}, 406(2):Paper No. 24, 20, 2025.

\bibitem{AZ1}
R.~Azencott and E.~N. Wilson.
\newblock Homogeneous manifolds with negative curvature. {I}.
\newblock {\em Trans. Amer. Math. Soc.}, 215:323--362, 1976.

\bibitem{AZ2}
R.~Azencott and E.~N. Wilson.
\newblock Homogeneous manifolds with negative curvature. {II}.
\newblock {\em Mem. Amer. Math. Soc.}, 8(178):iii+102, 1976.

\bibitem{BT98}
A.~Bahri and I.~A. Taimanov.
\newblock Periodic orbits in magnetic fields and {R}icci curvature of
  {L}agrangian systems.
\newblock {\em Trans. Amer. Math. Soc.}, 350(7):2697--2717, 1998.

\bibitem{BRCF_2005}
M.~Barros, A.~Romero, J.~L. Cabrerizo, and M.~Fern\'andez.
\newblock The {G}auss-{L}andau-{H}all problem on {R}iemannian surfaces.
\newblock {\em J. Math. Phys.}, 46(11):112905, 15, 2005.

\bibitem{beaufort2026tensortomographyframeflow}
L.-B. Beaufort.
\newblock Tensor tomography and frame flow ergodicity for magnetic flows in
  higher dimensions, 2026.
\newblock Preprint available at {\url{https://arxiv.org/abs/2604.12495}}.

\bibitem{BK_2022}
G.~Benedetti and J.~Kang.
\newblock On a systolic inequality for closed magnetic geodesics on surfaces.
\newblock {\em J. Symplectic Geom.}, 20(1):99--134, 2022.

\bibitem{Besse}
A.~L. Besse.
\newblock {\em Einstein manifolds}, volume~10 of {\em Ergebnisse der Mathematik
  und ihrer Grenzgebiete (3) [Results in Mathematics and Related Areas (3)]}.
\newblock Springer-Verlag, Berlin, 1987.

\bibitem{BJ_2008}
A.~V. Bolsinov and B.~Jovanovi\'c.
\newblock Magnetic flows on homogeneous spaces.
\newblock {\em Comment. Math. Helv.}, 83(3):679--700, 2008.

\bibitem{brigant2025conjugatepointsliegroups}
A.~Le Brigant, L.~Lichtenfelz, and S.~C. Preston.
\newblock Conjugate points on {L}ie groups with left-invariant metrics, 2025.
\newblock Preprint available at {\url{https://arxiv.org/abs/2408.03854}}.

\bibitem{Burns-Paternain_2002}
K.~Burns and G.~P. Paternain.
\newblock Anosov magnetic flows, critical values and topological entropy.
\newblock {\em Nonlinearity}, 15(2):281--314, 2002.

\bibitem{Cheeger-Ebin}
J.~Cheeger and D.~G. Ebin.
\newblock {\em Comparison theorems in {R}iemannian geometry}, volume Vol. 9 of
  {\em North-Holland Mathematical Library}.
\newblock North-Holland Publishing Co., Amsterdam-Oxford; American Elsevier
  Publishing Co., Inc., New York, 1975.

\bibitem{Paternain_2007}
N.~S. Dairbekov, G.~P. Paternain, P.~Stefanov, and G.~Uhlmann.
\newblock The boundary rigidity problem in the presence of a magnetic field.
\newblock {\em Adv. Math.}, 216(2):535--609, 2007.

\bibitem{DGJ_2025}
V.~Dragovi\'c, B.~Gaji\'c, and B.~Jovanovi\'c.
\newblock Integrability of homogeneous exact magnetic flows on spheres.
\newblock {\em Regul. Chaotic Dyn.}, 30(4):582--597, 2025.

\bibitem{Efimov_2005}
D.~I. Efimov.
\newblock The magnetic geodesic flow on a homogeneous symplectic manifold.
\newblock {\em Sibirsk. Mat. Zh.}, 46(1):106--118, 2005.

\bibitem{EGM_2021}
J.~Epstein, R.~Gornet, and M.~B. Mast.
\newblock Periodic magnetic geodesics on {H}eisenberg manifolds.
\newblock {\em Ann. Global Anal. Geom.}, 60(3):647--685, 2021.

\bibitem{Gouda-97}
N.~Gouda.
\newblock Magnetic flows of {A}nosov type.
\newblock {\em Tohoku Math. J. (2)}, 49(2):165--183, 1997.

\bibitem{Heintze_1974}
E.~Heintze.
\newblock On homogeneous manifolds of negative curvature.
\newblock {\em Math. Ann.}, 211:23--34, 1974.

\bibitem{Helgason}
S.~Helgason.
\newblock {\em Differential geometry, {L}ie groups, and symmetric spaces},
  volume~34 of {\em Graduate Studies in Mathematics}.
\newblock American Mathematical Society, Providence, RI, 2001.
\newblock Corrected reprint of the 1978 original.

\bibitem{Kobayashi_1962}
S.~Kobayashi.
\newblock Homogeneous {R}iemannian manifolds of negative curvature.
\newblock {\em Tohoku Math. J. (2)}, 14:413--415, 1962.

\bibitem{KN2}
S.~Kobayashi and K.~Nomizu.
\newblock {\em Foundations of differential geometry. {V}ol. {II}}, volume Vol.
  II of {\em Interscience Tracts in Pure and Applied Mathematics, No. 15}.
\newblock Interscience Publishers John Wiley \& Sons, Inc., New
  York-London-Sydney, 1969.

\bibitem{KN1}
S.~Kobayashi and K.~Nomizu.
\newblock {\em Foundations of differential geometry. {V}ol. {I}}.
\newblock Wiley Classics Library. John Wiley \& Sons, Inc., New York, 1996.
\newblock Reprint of the 1963 original, A Wiley-Interscience Publication.

\bibitem{MSY_2008}
A.~A. Magaz\"ev, I.~V. Shirokov, and Y.~A. Yurevich.
\newblock Integrable magnetic geodesic flows on {L}ie groups.
\newblock {\em Teoret. Mat. Fiz.}, 156(2):189--206, 2008.

\bibitem{Maier_2026}
L.~Maier.
\newblock On geometric hydrodynamics and infinite-dimensional magnetic systems.
\newblock {\em Nonlinearity}, 39(5):Paper No. 055002, 16, 2026.

\bibitem{SIGMA-26}
J.~Marshall~Reber and I.~Terek.
\newblock Finiteness of totally magnetic hypersurfaces.
\newblock {\em SIGMA Symmetry Integrability Geom. Methods Appl.}, 22:Paper No.
  043, 15, 2026.

\bibitem{Merry_2010}
W.~J. Merry.
\newblock Closed orbits of a charge in a weakly exact magnetic field.
\newblock {\em Pac. J. Math.}, 247(1):189--212, 2010.

\bibitem{Milnor_1976}
J.~Milnor.
\newblock Curvatures of left invariant metrics on {L}ie groups.
\newblock {\em Advances in Math.}, 21(3):293--329, 1976.

\bibitem{Moroianu-Ovando}
A.~Moroianu and G.~P. Ovando.
\newblock Magnetic {K}illing tensors and first integrals of the magnetic flow.
\newblock {\em Lett. Math. Phys.}, 116(3):Paper No. 65, 26, 2026.

\bibitem{MN_2022}
M.~I. Munteanu and A.~I. Nistor.
\newblock Magnetic curves in the generalized {H}eisenberg group.
\newblock {\em Nonlinear Anal.}, 214:Paper No. 112571, 18, 2022.

\bibitem{Ovando-Subils_2023}
G.~P. Ovando and M.~Subils.
\newblock Magnetic trajectories on 2-step nilmanifolds.
\newblock {\em J. Geom. Anal.}, 33(6):Paper No. 186, 28, 2023.

\bibitem{AGAG-26}
I.~Terek.
\newblock The submanifold compatibility equations in magnetic geometry.
\newblock {\em Ann. Global Anal. Geom.}, 69(2):Paper No. 4, 15, 2026.

\bibitem{Turhan_2020}
T.~Turhan.
\newblock Magnetic trajectories in three-dimensional {L}ie groups.
\newblock {\em Math. Methods Appl. Sci.}, 43(5):2747--2758, 2020.

\bibitem{Wojtkowski_2000}
M.~P. Wojtkowski.
\newblock Magnetic flows and {G}aussian thermostats on manifolds of negative
  curvature.
\newblock {\em Fund. Math.}, 163(2):177--191, 2000.

\bibitem{Wolf_1964}
J.~A. Wolf.
\newblock Homogeneity and bounded isometries in manifolds of negative
  curvature.
\newblock {\em Illinois J. Math.}, 8:14--18, 1964.

\end{thebibliography}
\bibliographystyle{plain}

\end{document}